\UseRawInputEncoding

\documentclass[12pt]{amsart}
\usepackage{amsfonts}
\usepackage{amsmath}
\usepackage{amssymb,latexsym}
\usepackage[mathcal]{eucal}
\usepackage{amscd}

\usepackage[pdftex,bookmarks,colorlinks,breaklinks]{hyperref}
\input xy
\xyoption{all}

\newcommand{\Acal}{\mathcal{A}}
\newcommand{\Bcal}{\mathcal{B}}

\newcommand{\Ucal}{\mathcal{U}}

\newcommand{\ch}{\mathbf{1}}

\newcommand{\R}{\mathbb{R}}

\newcommand{\N}{\mathbb{N}}

\newcommand{\al}{\alpha}

\newcommand{\del}{\delta}

\newcommand{\ep}{\epsilon}

\newcommand{\la}{\lambda}

\newcommand{\om}{\omega}
\newcommand{\Om}{\Omega}

\newcommand{\ol}{\overline}

\newcommand{\br}{\vspace{3 mm}}
\newcommand{\imp}{\Rightarrow}

\newcommand{\rest}{\upharpoonright}

\newcommand{\cls}{{\rm{cls\,}}}

\newcommand{\conv}{{{\rm{conv}\,}}}
\newcommand{\ext}{{\rm{ext\,}}}

\newcommand{\card}{{\rm{card\,}}}

\newcommand{\spann}{{\rm{span}}}

\newcommand{\Iso}{{\rm{Iso\,}}}

\newcommand{\bary}{{\rm{bary}}}

\newcommand{\ben}{\begin{enumerate}}
\newcommand{\een}{\end{enumerate}}

\newcommand{\vertiii}[1]{{\left\vert\kern-0.25ex\left\vert\kern-0.25ex\left\vert #1 
    \right\vert\kern-0.25ex\right\vert\kern-0.25ex\right\vert}}

\swapnumbers
\theoremstyle{plain}
\newtheorem{thm}{Theorem}[section]
\newtheorem{cor}[thm]{Corollary}
\newtheorem{lem}[thm]{Lemma}

\theoremstyle{definition}
\newtheorem{defn}[thm]{Definition}

\newtheorem{rmk}[thm]{Remark}

\newtheorem{exa}[thm]{Example}

\newtheorem{claim}[thm]{Claim}

\newtheorem{question}[thm]{Question}

\begin{document}

\title[On a question of Kazhdan and Yom Din]
{On a question of Kazhdan and Yom Din\\
\tiny{With an appendix by Nicolas Monod}}

\date{December 23, 2021}

\author{Eli Glasner}

\address{Department of Mathematics\\
     Tel Aviv University\\
         Tel Aviv\\
         Israel}
\email{glasner@math.tau.ac.il}
%
%

\thanks{Nicola Monod thanks Eli Glasner for his comments and for welcoming his appendix.}
\begin{abstract}
Studying a question of Kazhdan and Yom Din we first prove some fixed point theorems
for linear actions of a discrete countable group $G$ on a dual Banach space $V^*$,
and then show that the answer to the question in full generality is in the negative.
In an appendix it is shown that the answer is negative also for the Banach space $\Bcal(H)$
of bounded linear operators on an infinite dimensional Hilbert space.
\end{abstract}



\thanks{This research was supported by a grant of the Israel Science Foundation (ISF 1194/19)}

%
%
%

\maketitle

\tableofcontents
\setcounter{secnumdepth}{2}


\section*{Introduction}

In an effort to prove ``an approximate Schur's lemma" David Kazhdan and 
Alexander Yom Din 
arrived at the following question, which they communicated to us (Benjamin Weiss and the author).

\br

\begin{question}\label{conj-K}
Is the following assertion true ?
\begin{quote}
Let $\del >0$ be given and let $V$ be a Banach space, equipped with a linear and isometric action 
 of a discrete group $G$. Let $\al \in V^*$ be such that  $\|\al \| =1$ and 
 $\| g\al - \al \| \le \frac{\del}{10}$ for all $g \in G$. Then there exists a $G$-invariant
 $\beta \in V^*$ such that $\|\al - \beta\| \le  \del$.
 \end{quote}
 \end{question}

 The most interesting and relevant Banach space for Kazhdan and Yom Din
 is the space $\Bcal(H)$, comprising the bounded linear operators on a Hilbert space
 and equipped with  its weak operator topology.
 (On norm bounded sets of $\Bcal(H)$, the weak operator topology and the weak$^*$ 
 (= ultraweak) topology coincide.)

 \br
 
As a motivation one can consider the following easy example. 
 
%

\begin{exa}\label{exa-linfty}
Consider the Banach space $V = l_1(G)$ (over $\R$) with dual $V^* = l_\infty(G)$, 
where the latter is equipped with its sup norm.
Clearly then the assertion in Question \ref{conj-K} holds in this case. 

In fact, the action of $G$ on $l_\infty(G)$ is by translations:
$$
(g \cdot \al)(h) = \al(hg), \quad \forall  \ g, h \in G, \  \al \in l_\infty(G),
$$ 
and a function $\al \in l_\infty(G)$ is $G$-invariant iff it is a constant function.
Now given $\al \in V^* = l_\infty(G)$ satisfying 
(i) $\|\al \|=1$ and (ii) $\| g\al - \al \| \le \frac{\del}{10}$ for all $g \in G$,
either the function $\ch \in \ell_\infty(G)$ is a $G$-fixed point with
$\|\ch - \al\|  \leq \del$, or the function $-\ch \in \ell_\infty(G)$ is a $G$-fixed point with
$\|\ch + \al\|  \leq \del$.

Since $\|\al \|=1$, for some $h \in G$ we have either $|\al(h) - 1| \leq \frac{\del}{10}$
or $|\al(h) + 1| \leq \frac{\del}{10}$. Assuming the first case,
if $\|\ch - \al\| > \del$, then for some $g\in G$ we have
$|\al(g) - \ch(1)| = |\al(g) - 1|> \del$.
Now,
\begin{align*}
 \del  & < |\al(g) - 1| \\
 &   \le |\al(g) - \al(h)| + | \al(h) -1| \\
&= |((h^{-1}g)\al)(h) -  \al(h)| + |\al(h) -1|\\
& \le \| (h^{-1}g)\al - \al\| +  |\al(h) -1|  \leq   2 \frac{\del}{10},
 \end{align*}
a contradiction. The second case is treated similarly
 \end{exa}

  \br
  
The  assertion of Question \ref{conj-K} certainly holds whenever the acting group $G$ is amenable.
However, there are many general fixed point theorems which can be applied 
 also in cases where the acting group is any topological group.
Most notable are the classical Ryll Nardzewski's fixed point theorem and its many 
generalizations (see e.g. \cite{GM-12}). 

The question \ref{conj-K} though introduces a new perspective on the question 
of the existence of fixed points for affine actions on compact convex sets, namely the
$\del$-norm assumption.

\br

Our main results in this work are as follows.

In Section \ref{sec-RN} we present a short proof of a Ryll Nardzewski type
theorem for norm separable, convex weak$^*$-compact sets, Theorem \ref{RN}.
We also show how results of Losert \cite{Lo} and Bader, Gelander and Monod \cite{BGM}
provide a positive answer to question \ref{conj-K} for the Banach space $V = C(X)$,
the space of continuous complex valued functions on a compact Hausdorff space $X$.
In Section \ref{sec-examples} we show how Theorem \ref{RN}
may fail when the norm separability assumption is removed.
In particular, these examples show that a general fixed point theorem for
convex weak$^*$ compact sets
fails for $V = \Bcal(H)$.
In Section \ref{sec-measures} we provide a new proof of the fact that 
Question \ref{conj-K} has an affirmative answer for the Banach space $V = C(X)$,
Theorem \ref{thm-CX}.

%
%
%
Finally, in Section \ref{sec-CE} we show that for $G = F_2$, the free group
on two generators, given $0 < \del <1$, there exists a separable Rosenthal Banach space $V$
\footnote{A Banach space is called {\em Rosenthal} if it does not contain an isomorphic copy
 of $\ell_1(\N)$.}, specially contrived for that purpose, 
and an affine representation of a minimal and strongly proximal
$G$ dynamical system $(X,G)$ on $V^*$, such that the resulting action of $G$ on $V^*$
has the following properties.
The only $G$ fixed point in $V^*$ is $0$, but
there exists an element $\xi \in V^*$, $\|\xi\|=1$ 
such that $\|g\xi - \xi\| \leq \del$ for every $g \in G$.
This shows that, in full generality, the assertion of Question \ref{conj-K} fails.

%
%

An important aspect of the present work is the application of the newly developed theory of 
representations of dynamical systems on Banach spaces. This theory was developed by
Michael Megrelishvili and the author in a series of works. We refer to the review article \cite{GM-R}
for more details.

I would like to thank David Kazhdan and Alexander Yom Din for addressing the question to us.
I thank Benjamin Weiss and Michael Megrelishvili whose invaluable advice greatly improved this work.
Also, thanks are due to the referee for many very helpful remarks.

\br 

When this work was accepted for publication I posted it on the arXiv \cite{arXiv}.
After reading my work there Nicolas Monod obtained a negative answer to  
the question of Kazhdan and Yom Din in the case of the space $\Bcal(H)$.
He then kindly agreed to publish it as an appendix to this paper.

\br

\section{Ryll Nardzewski theorem for norm separable $w^*$-compact convex sets}\label{sec-RN}

\begin{defn}\label{def-sp}
A dynamical system $(X,G)$ is called {\em proximal} if 
for every pair of points $x, y \in X$ there is a net $g_i \in G$
and a point $z \in X$ such that $\lim g_i x = \lim g_i y =z$;
it is called {\em strongly proximal} if the induced action
on the space $M(X)$ of Borel probability measures, equipped with its weak$^*$ topology,
is proximal, or equivalently, when for every $\mu \in M(G)$ there is a net $g_i \in G$
and a point $x \in X$ such that $\lim g_i \mu = \del_x$.
Both proximality and strong proximality are
preserved under surjective homomorphisms of dynamical systems. For more details see \cite{Gl-prox}. 
\end{defn}

The following theorem is not new (see \cite[Theorem 2.4]{V},  
\cite[Theorem 15]{N-P} and also \cite{GM-12}); however, our proof
adapted from \cite[Chapter III, Theorem 5.2]{Gl-prox}, is more transparent
and much shorter.

\begin{thm}\label{RN}
Let $V$ be a Banach space and $K$ a subset of $V^*$ which is 
(i) {\bf norm separable}, (ii) convex, (iii) closed in the weak$^*$ topology.
Let $G < Iso(V)$ be a group of linear isometries such that $gx \in K$ for every $g \in G$ and $x \in K$
(with respect to the natural action of $G$ on $V^*$).
Then there is a point $x_0 \in K$ such that $gx_0 =x_0$ for every $g \in G$.
\footnote{Nicolas Monod brought  to my attention the fact that this result can be found in Bourbaki's TVS book \cite[EVT IV.41, case (c)]{Bour}.}
\end{thm}

\begin{proof} 
We can assume that the affine dynamical system $(K,G)$ is irreducible  (i.e.
if $\emptyset \not = Q \subset K$ is $G$-invariant, convex and weak$^*$ closed, then $Q = K$).

Let $B$ be a norm closed ball of radius $1$ centered at $0 \in V^*$.
Denote  $X = \ol{\ext (K)}$ (the $w^*$-closure of the set of extreme points).

By the separability of $K$, given $\ep >0$ there exists a set $\{x_i\}_{i=1}^\infty$ of points in $X$ such that 
$\{x_i + \ep B\}_{i=1}^\infty$ is a denumerable cover of $X$. 
We claim that $B$ is weak$^*$ closed. In fact, if $y \not \in B$
there is a vector $v  \in V$ such that 
$r = y(v) > 1$ and then 
$$
\{z \in V^* : z(v)  > \frac{r +1}{2}\} 
$$
is a weak$^*$ open neighborhood of $y$ which is disjoint from $K$.

By Baire's category theorem there is an $i$ and a non-empty weak$^*$ open set $W$ such that 
$$
W \cap X \subset (x_i + \ep B) \cap X.
$$

Since $K$ is irreducible, $X$ is the unique minimal set of the system $(K,G)$
and it is strongly proximal (\cite[Chapter III, Theorem 2.3]{Gl-prox}, see Definition \ref{def-sp} below).
If $x, y \in X$ we can find $g \in G$ such that $gx$ and $gy$ are in $X \cap W \subset x_i + \ep B$.
Since $\ep$ was arbitrary, this contradicts the fact
that the elements of $G$ act by norm isometries, unless $X$ and hence also $K$ are trivial
one point sets.
\end{proof}

\br

Recall that a Banach space $V$ is an {\em Asplund\/} space if the
dual of every separable linear subspace of $V$ is separable (iff $V^*$ has the Radon-Nikod\'ym
property). 
Reflexive Banach spaces and spaces of the type $c_0(\Gamma)$ are Asplund. 
For more details see e.g. \cite{Fa}.

\begin{cor}
Let $V$ be an Asplund Banach space and $K$ a subset of $V^*$ which is 
 (i) convex, (ii) closed in the weak$^*$ topology.
Let $G < Iso(V)$ be a countable group of linear isometries such that $gx \in K$ for every $g \in G$ and $x \in K$
(with respect to the natural action of $G$ on $V^*$).
Then there is a point $x_0 \in K$ such that $gx_0 =x_0$ for every $g \in G$.
\end{cor}

\begin{proof}
In the case that $V$ is a separable Banach space the assumption (i) in Theorem \ref{RN} is redundant.
In the general case we argue as follows. First observe that, as we assume that $G$ is a countable group,
every element of $V$ is contained in a $G$-invariant separable subspace of $V$.
Let $V_i,\ i \in I$, be the collection of such subspaces, partially ordered by inclusion.
We then have for each $i \in I$, that $V_i^* \cong V^*/V_i^\perp$ (see e.g \cite[Proposition 2.7]{FHH}).
We let $K_i$ be the canonical image of $K$ in $V_i^*$, and observe that
$K \cong \lim_{\leftarrow} K_i$. Let $K_0 \subset K$ be an irreducible subset
(i.e. a $G$-invariant, weak$^*$-closed,  convex subset which is minimal with respect to these properties;
such a subset always exists by Zorn's lemma, see \cite[Section III.2]{Gl-prox}).
Now observe that by the Asplund property the image of $K_0$ in each $V_i^*$ is a single point
and thus we finally conclude that $K_0$ itself is a singleton whose unique 
element is the required fixed point. 
\end{proof}

\br

Since on norm bounded sets of $\Bcal(H)$ the weak operator topology and the weak$^*$ topology coincide
we also have this theorem in the context of $\Bcal(H)$.

\begin{thm}\label{RNO}
Let $H$ be a separable Hilbert space and $K$ a subset of $\Bcal(H)$ which is 
(i) {\bf norm separable}, (ii) convex, (iii) closed in the weak operator topology.
Let $G < \Ucal(H)$ be a group of unitary operators such that $gx \in K$ for every $g \in G$ and $x \in K$.
Then there is a $G$ fixed point in $K$.
\end{thm}

\br

Another application of the Ryll Nardzewski theorem is the following theorem which is a special case of
\cite[Theorem A]{BGM} and \cite{Lo}.

\begin{defn}
A Banach space $V$ is said to be {\em  L-embedded}  
if its bidual space $V^{**}$ can be decomposed as 
$V^{**}= V \oplus_1 V_0$ for some $V_0 \subset V^{**}$
(where $\oplus_1$ indicates that the norm is the sum of the norms on $V$ and $V_0$).
\end{defn}

As mentioned in \cite{BGM},
by the Yosida-Hewitt decomposition \cite{YH}, every $L_1$ space and 
more generally every predual of any von Neumann algebra \cite[Theorem III, 2.14]{T}
is L-embedded; in particular, this is the case for the dual of any $C^*$-algebra
(see \cite[Theorem 1.17.2]{Sa}).

\begin{thm}\cite[Theorem A]{BGM}\label{thm-BGM}
Let $A$ be a non-empty bounded subset of an $L$-embedded Banach space $V$.
Then there is a point in $V$ fixed by every linear isometry of $V$ preserving $A$. 
Moreover, one can choose a fixed point which minimises $\sup_{a \in A} \|v - a\|$ over all $v \in V$.
\end{thm}

Thus if $V$ is an $L$-embedded Banach space,  $G$ is a group of linear isometries
of $G$, and for some $a_0  \in V$ and $\del > 0$ we have
$\|ga_0 - a_0\| \le \del, \ \forall g \in G$, then with $A = \{ga_0 : g \in G\}$,
by the theorem above, there is a $G$-fixed point $b \in V$ 
with $\|b - a_0\| \leq  \sup_{a \in A} \|b - a\|  \leq \sup_{a \in A} \|a_0 - a\| \leq \del$.
This confirms the assertion of Question \ref{conj-K} for $L$-embedded Banach spaces.

\begin{thm}\label{thm-BGMc}
The assertion of Question \ref{conj-K} holds for every  $L$-embedded Banach space $V$.
\end{thm}

In particular since  $C(X)^*$ is the dual of the (commutative) $C^*$-algebra $C(X)$ and 
hence the predual of a von Neumann algebra, 
this confirms the assertion of Question 
\ref{conj-K} for $V = C(X)$. In Section \ref{sec-measures} we will provide a new proof of this theorem.

\br

\section{Examples where $K$ is not norm separable}\label{sec-examples}

As we will now see,
the assumption that the set $K$ (or $X$) is norm separable 
in Theorems \ref{RN} and \ref{RNO} is essential.

\br

When one considers $\ell_\infty(G)$ as a subspace of $\Bcal(H)$, with $H = \ell_2(G)$,
and elements $x \in \ell_\infty(G)$ are viewed as multiplication operators, 
it is easy to check that a subset $X \subset [0,1]^G$,
is closed in $[0,1]^G$
iff $X$ is closed in the weak$^*$ topology on
$\ell_\infty(G)$ (induced by $\ell_1(G)$), iff it is closed in the weak operator topology in $\Bcal(H)$.

When $X$ is a Bebutov subshift; i.e. a closed subset of $[0,1]^G$ invariant under left translations,
the induced Koopman unitary representation $\pi : G \to \Ucal(H)$ defines an associated
representation $\rho : G \to \Bcal(H)$, $\rho(g) T : = \pi(g) \circ T \circ \pi(g)^{-1}$.
Note that, when $X \subset \{0,1\}^G$ then,
viewed as a subset of $\Bcal(H)$, it consists of projections.

\br

Hereditarily almost equicontinuous systems (HAE for short) were defined and studied in \cite{GM}

\begin{defn}\label{def-HAE}
Given a (compact metric) dynamical system $(X,G)$,
a point $x \in X$ is called an {\em equicontinuity point}
if for every $\ep >0$ there is a $\del >0$ such that
$d(x,x') <\del \imp d(gx, gx') < \ep$ for every $g \in G$.
When the set $Eq(X)$ of equicontinuity points is dense in $X$ we say that the system $(X,T)$
is {\em almost equicontinuous} and we say that 
it is {\em hereditarily almost equicontinuous} (HAE for short)
if every subsystem $Y \subset X$ is almost equicontinuous.
\end{defn}

\begin{rmk}
A (not necessarily metric) dynamical system $(X,G)$ is called {\em not sensitive} when 
there exists an entourage $\ep$ (= a neighborhood of the diagonal in $X \times X$) such that for
every $x\in X$ and any neighborhood $U$ of $x$ there exists $y\in 
U$ and $g\in G$ such that $(gx,gy) \notin \ep$.
It is called {\em hereditarily non sensitive} (HNS for short)
when every subsystem $Y$ of $X$ is not sensitive.
 For metrizable systems the properties HNS and HAE coincide.  For more details see \cite{GM}.
\end{rmk}

The next theorem is a special case of \cite{V}, and \cite[Theorem 9.14 and Corollary 10.4]{GM}.

\begin{thm}
Let $G$ be a discrete countable group and let $X \subset [0,1]^G$ be a closed invariant set.
The system $(X,G)$ is HAE  iff the set $X$, considered as a subset
of $\ell_\infty(G)$, is norm separable.
If $X \subset \{0,1\}^G$ then it is HAE iff it is countable.
\end{thm}

Thus it follows that whenever the dynamical system $(X,G)$ (with $X \subset \{0,1\}^G$),
is not HAE (e.g. this is the case when it is weakly mixing or strongly proximal), 
the set $X \subset \ell_\infty(G)$  is not norm separable.

\begin{thm}
Let $G$ be a discrete countable non-amenable group. Then there is an infinite
closed invariant $X \subset \{0,1\}^G$ with the following properties:
\begin{enumerate}
\item
The system $(X,G)$ is minimal and strongly proximal (hence in particular not HAE). 
\item
The space $X$ considered as a subspace of $\Bcal(H)$ is not norm separable, and
\item
The weak operator closed, convex, $G$-invariant set
$$
Q := WO{\text{-}}\cls (\conv (X)) \subset \Bcal(H),
$$
is a strongly proximal affine $G$-system and therefore admits no fixed point.
\end{enumerate}
\end{thm}

\begin{proof}
It is well known that $G$ is non-amenable iff there is an infinite minimal strongly proximal 
$G$-system, \cite[Chapter III, Theorem 3.1]{Gl-prox}. Since $G$ is discrete, 
general well known considerations show that such a system can be found with $X \subset \{0,1\}^G$.
There are several ways to see this; one short proof is as follows.
Consider the universal minimal strongly proximal $G$-system $(\Pi_s(G),G)$. Since $G$ is discrete
the space $\Pi(G)$ is totally disconnected (in fact an extremely disconnected space, see e.g. \cite{GTWZ}).
Now let $O \subsetneq \Pi_s(G)$ be a nonempty clopen subset and let
$\pi : \Pi_s(G) \to \{0,1\}^G$ be the corresponding ``name map":
$\pi(z)(g) = \ch_O(gz),\ z \in \Pi_s(G), \ g \in G$.
The image system $(X,G)$, where $X = \pi((\Pi_s(G))$, is the required system.
The rest of the claims follow from the above discussion (see Example \ref{exa-DM} below for
a concrete example of such a system for $F_2$).
\end{proof}

\br

\begin{question}
Which weak operator topology closed convex subsets of $\Bcal(H)$ are norm separable ?
\end{question}

\br

\begin{exa}\label{exa-DM}
The free group $G = F_2$ on two generators, say $a$ and $b$, is a hyperbolic group and its 
boundary (the boundary of its Cayley graph) 
can be identified with the compact metric space $\Om$ (a Cantor set) of all
the one-sided infinite reduced words $\om$ on the symbols $a, b, a^{-1}, b^{-1}$.
The group action on $\Om$ is given by
$$
F_2 \times \Om \to \Om,   \quad (g, \om) = g \cdot \om,
$$
where $g \cdot \om$ is obtained by concatenation of $g$ (written in its reduced form)
and $\om$ and then performing the needed cancelations.
The resulting dynamical system is minimal and strongly proximal. 
Taking any nontrivial clopen partition $\Om = P_0 \sqcup P_1$,
and defining the corresponding {\em name map}
$$
\phi \colon \Om \to \{0,1\}^G, \quad \phi(\om)(g) = \ep \ {\text {when}} \ g \cdot \om \in P_\ep,
$$
we obtain an infinite subshift $X = \phi(\Om) \subset \{0,1\}^G$ which is
minimal and strongly proximal.

It turns out that the systems $(\Om,G)$ and $(X,G)$ are {\em tame} (see 
Section \ref{sec-CE} below and \cite[Example 6.7]{GM-T}).
\end{exa}

\br

\section{A fixed point theorem for measures}\label{sec-measures}

In this section we provide a direct proof of the assertion of Question \ref{conj-K} for the case $V = C(X)$,
the Banach space of continuous real valued functions on a compact metric space $X$
equipped with the sup norm, and actions on $C(X)$ coming from actions by homeomorphisms on $X$ 
(see Theorem \ref{thm-BGMc} above and the remark which follows it).

 \br
 
Let $(X,G)$ be a compact metric dynamical system. 
Let $M(X) \subset C(X)^*$ denote the 
convex $w^*$-compact set of Borel probability measures on $X$.
We use $\|\cdot\|$  to denote both the sup norm on $C(X)$ and the total variation norm on 
its dual $C(X)^*$ (viewed as the space of signed measures).

\br

We will use the following well known lemma.

\begin{lem}\label{affi}
For $\al, \beta \in M(X)$ we have the identity
$$
\frac12 \| \al - \beta \| + \| \al \wedge \beta \| = 1.
$$
Thus $\al \perp \beta$ iff $\| \al - \beta\|=2$.
\end{lem}

%
%

\br

\begin{thm}\label{thm-M}
Suppose $\al \in M(X)$ is
such that, for a positive $ \del < 1 $,
$\| g\al - \al \| \le \frac{\del}{10}$ for all $g \in G$. Then there exists a 
non zero $G$-invariant measure
 $\mu_0$ such that $\|\al - \mu_0\| \le  \del$.
\end{thm}

\begin{proof}
Let $Q_{\del/2} = \{\beta \in C(X)^* : \|\al - \beta\| \leq \del/2\}$,
the ball of radius $\del/2$ centered at $\al$.
Let 
$$
Q : = w^*{\text{-}}\cls\left(\conv(\{g \al : g \in G\})\right) \subset M(X) \cap Q_{\del/2}.
$$
Note that $\|\theta - \eta\| \leq \del$ for any pair $\theta, \eta \in Q$.
Clearly $Q$ is a nonempty, $w^*$-closed, convex, $G$-invariant subset of $M(X)$.
%


Let $m$ be a symmetric probability measure on $G$ with full support.
Let $P_m : M(X) \to M(X)$ be the corresponding Markov operator, defined by convolution with $m$:
$$
P_m(\theta) = m*\theta, \quad  \theta \in M(X).
$$
Clearly $P_m(Q) \subset Q$.
Let $\la \in M(M(X))$ be any weak$^*$ limit point of the sequence
$\frac1N \sum _{j =0}^{N-1}P_m^j \del_{\al}$,
say
$$
\la = \lim_{k \to \infty}\frac{1}{N_k} \sum _{j =0}^{N_k-1}P_m^j \del_{\al}.
$$
It follows that $\la \in M(Q)$ and that
the measure $\la$ is $m$-stationary; i.e. $m*\la = \la$.
Let $\bary \colon M(M(X)) \to M(X)$ denote the weak$^*$-continuous, affine, barycenter map
and set 
$$
 \bary(\la) = \int_Q \theta d\la(\theta) := \mu \in Q.
$$
The map $\bary$ is an affine homomorphism of dynamical systems and we have $m*\mu =\mu$.
Thus the non-singular dynamical systems 
$(Q, \la; G, m)$ and $(X, \mu; G, m)$ are $m$-stationary systems (see \cite{F-G}).

%
%
%
%
%

\br

For each $\theta \in Q$ we let
$$
\theta = \theta_a + \theta_d, \qquad \theta_a \ll \mu, \quad  \theta_d \perp \mu,
$$
be the unique Lebesgue decomposition of $\theta$ with respect to $\mu$.

By Lemma \ref{affi}, for every $\theta \in Q$,
$$
0 < 1 - \frac12 \del \leq 1 - \frac12 \|\theta - \mu \| 
= \| \theta \wedge \mu\| = \int \min(\ch_X, \frac{d\theta_a}{d\mu}) d\mu
\leq \int \frac{d\theta_a}{d\mu} d\mu = \|\theta_a\|.
$$

Now for $g \in G$ we have
$$
g\theta = g\theta_a + g\theta_d = (g\theta)_a + (g\theta)_d,
$$
and
$$
g\theta_a \ll g \mu \cong \mu, \qquad \ g\theta_d \perp g\mu \cong \mu.
$$
Hence, by the uniqueness of the decomposition, we have
\begin{equation}\label{intertwine}
g\theta_a = (g\theta)_a, \qquad g\theta_d = (g\theta)_d.
\end{equation}

We write
$$
F_\theta = \frac{d\theta_a}{d\mu} \quad {\text{and}} \quad 
H_\theta = \sqrt{F_\theta}.
$$
Note that $F_\mu = H_\mu = \ch_X$.

We consider the usual unitary representation of $G$ on
$L_2(X,\mu)$ given by
$$
U_gf(x)=f(g^{-1}x)u(g^{-1},x),\qquad \text{\rm with}\qquad
u(g,x)=\sqrt{\frac{dg^{-1}\mu}{d\mu}}.
$$

Then, for $g\in G, \ \theta \in Q$ and $f \in L_2(X,\mu)$,
denoting $v(g,x)=\frac{dg^{-1}\mu}{d\mu}$,
we get
\begin{align*}
\int f(x) d(g\theta_{a}) & = \int f(x) F_{g\theta}(x) d\mu\\
&= \int  f(x)dg\theta_a\\
&=\int f(gx)F_\theta(x)d \mu\\
&= \int f(x)F_\theta(g^{-1}x)dg\mu\\
&=\int f(x)F_\theta (g^{-1}x)v(g^{-1},x)d\mu.
\end{align*}
Hence  $F_{g\theta}=(F_\theta\circ g^{-1})\cdot v(g^{-1},\cdot)$ and
\begin{equation}\label{U}
H_{g\theta}=(H_\theta\circ g^{-1})\cdot u(g^{-1},\cdot)=U_g H_\theta.
\end{equation}
Let 
$\phi : Q \to  C(X)^*$ be defined by $\phi(\theta) = \theta_a = F_\theta d\mu$.
 Let
 $$
 Q_a := \phi(Q), 
\quad \la_a := \phi_*(\la).
 $$
 By equation (\ref{intertwine}), $\phi$ intertwines the actions of $G$ on $Q$ and $Q_a$.

Let $J : Q_a \to L_2(X,\mu)$ be defined by $J(\theta_a) = H_\theta$,
and set
$$
\tilde{Q} := J(Q_a), 
\quad \tilde{\mu} := J_*(\la_a).
$$

Clearly  the map $J : (Q_a, G) \to (\tilde{Q}, G)$ is a measurable isomorphism 
and by equation (\ref{U}) 
$$
J(g\theta_a) = H_{g \theta} = U_g H_\theta = U_g J(\theta_a).
$$ 
Thus the map
$$
J 
\colon (Q_a, \la_a; G, m) \to (\tilde{Q}, \tilde{\mu}; G, m)
$$
is an isomorphism of $m$-stationary systems,
where again the action of $G$ on $\tilde{Q} \subset L_2(X,\mu)$ is via the unitary representation
 $g\mapsto U_g$ .

Now the dynamical system $(\tilde{Q}, \tilde{\mu}; G,m)$ is WAP (weakly almost periodic) and $m$-stationary.
By \cite[Theorem 7.4]{F-G}, such a system is {\em $m$-stiff};  that is,
every $m$-stationary measure is actually invariant.
We therefore conclude that the $m$-stationary measure $\tilde{\mu}$ is $G$-invariant.
%
%

As the map $J$ is an isomorphism, this implies that also $\la_a$  is $G$-invariant.
Finally, applying the barycenter map
$$
\bary \colon M(Q_a) \to Q_{\del/2}
$$
and denoting $\mu_0 := \bary(\la_a) \in Q_{\del/2}$,
we conclude that $\mu_0$ is $G$-invariant 
with
$$
\| \al - \mu_0 \| \leq \del,
$$
and our proof is complete.
\end{proof}

\br

\begin{rmk}
Note that it may happen that we start with $\mu = \mu_0 + \theta$, where $\theta \perp \mu_0$,
$\|\theta\| < \del$, $\mu_0$ is $G$-invariant, and $\theta$ is
such that $w^*{\text{-}}\cls(G \theta)$ admits an $m$-stationary measure
which is not $G$-invariant.
\end{rmk}

\begin{rmk}
Normalizing $\mu_0$ we obtain the existence of a $G$-invariant probability measure.
\end{rmk}

A slight modification of the proof will yield the following.

\begin{thm}\label{thm-prob}
If $\al \in M(X)$ is such that for some $\del >0$,
$\| g\al - \al \| \leq 2 - \del$ for all $g \in G$, then there exists a 
$G$-invariant probability measure. 
\end{thm}

\begin{rmk}
Note the apparent similarity between Theorem \ref{thm-prob} and Proposition 4.14.(1) in \cite{BBHP}. 
\end{rmk}

%
%

\br

\begin{thm}\label{thm-CX}
Let $(X,G)$ be a compact $G$-space and suppose that $\mu \in C(X)^*$ is a signed measure
such that, for a positive $0 < \del < 1$,
$\| g\mu - \mu \| \le \frac{\del}{10}$ for all $g \in G$. Then there exists a $G$-invariant signed measure
 $\la$ such that $\|\mu - \la\| \le  2\del$.
 \end{thm}

\begin{proof}
Let $\mu = \mu^+ - \mu^-$ be the Jordan decomposition of $\mu$.
It is not hard to see that then also $\| g\mu^\pm - \mu^\pm \| \le \frac{\del}{10}$ for all $g \in G$.
Denoting $\mu^\pm_1 = \frac{\mu^\pm}{\|\mu^\pm\|} \in M(X)$ we have
$$
\|g \mu^\pm_1 - \mu^\pm_1\| \le \frac{\del}{10\|\mu^\pm\|}
$$ 
and, by Theorem \ref{thm-M} there are $G$-invariant positive measures $\la^\pm_1$ with
$\|\la^\pm_1 - \mu^\pm_1\| \le \frac{\del}{\|\mu^\pm\|}$.
Finally, with $\la^\pm = \|\mu^\pm_1\|\la^\pm_1$ and $\la = \la^+ - \la^-$, we conclude that
$\la$ is a $G$-invariant signed measure satisfying
$$
\| \la - \mu \| \le 2\del.
$$
\end{proof}

\br

 \section{A counter example to the $\del$ question}\label{sec-CE}
%
%
%

 In this section we will show how to construct a counter example to the assertion of Question \ref{conj-K}
 for $G =F_2$, the free group on two generators.
 In fact, the same construction will work for any discrete countable group that admits an effective
 minimal, strongly proximal, tame dynamical system.
 The idea is to start with such a minimal metric dynamical system $(X,G)$ and then,
 as in \cite{GM-rose}, by way of the Davis-Figiel-Johnson-Pelczy\'nski
 (DFJP)  construction \cite{DFJP}, to modify its 
 natural representation on $C(X)$ 
 in order to create a Rosenthal Banach space $V$ 
 and a representation of $(X,G)$ on $V^*$,
 so that in this representation the question \ref{conj-K} is refuted.
 For concreteness we will consider the $F_2$ dynamical system from Example \ref{exa-DM}.

We start with some background on enveloping semigroups, on tame dynamical systems
and on representations of dynamical systems on Banach spaces.
For simplicity we will assume that our dynamical systems are metrizable.
For more details see e.g. \cite{GM-rose}.

 \br
 
 \begin{defn}
The enveloping semigroup of the system $(X,G)$, denoted by $E(X)$,
 is defined as the pointwise closure in $X^X$ of the set of $g$-translations, $g \in G$.
 \end{defn} 
 
 \begin{defn} \label{d:Ros-F}\ 
 \begin{enumerate}
 \item
 A compact space $K$ is called {\em Rosenthal compact} if, for some Polish space $X$, it 
 can be homeomorphically embedded in the space $B_1(X)$ of real valued Baire class 1
 functions on $X$, endowed with the pointwise convergence topology.
 \item
Let $X$ be a compact topological space. We say that a subset $F\subset
C(X)$ is a \emph{Rosenthal family} (for $X$) if $F$ is norm bounded and
the pointwise closure $\cls_p(F)$ of $F$ in $\R^X$ consists of Baire class 1 functions.
\item
A dynamical system $(X,G)$ is called {\em tame} if 
for every $f \in C(X)$ the orbit $\{f \circ g : g \in G\}$ is a 
Rosenthal family.
\end{enumerate}
\end{defn}

\br

The dynamical Bourgain-Fremlin-Talagrand dichotomy (see \cite{BFT}) marks a sharp division in the 
domain of dynamical systems into two classes:  ``tame" and ``wild". 
It was first introduced by K\"{o}hler \cite{Ko}, and then was further developed 
is a series of papers by Glasner and Megrelishvili, Kerr and Li, and many other authors
(see \cite{GM}, \cite{KL}, \cite{GM-R}). 
The following theorem is from \cite[theorem 3.2]{GM}.

\begin{thm}[The dynamical BFT dichotomy]\label{f:D-BFT}
Let $X$ be a compact metric dynamical $G$-system and let $E(X)$ be its
enveloping semigroup. We have the following dichotomy. Either
\begin{enumerate}
\item
$E(X)$ is a separable Rosenthal compactum, hence with cardinality 
${\card}{E(X)} \leq 2^{\aleph_0}$; \ or
\item
$E(X)$
contains a homeomorphic copy of $\beta\N$
(the Stone-\u{C}ech compactification of $\N$),
hence ${\card}{E(X)} = 2^{2^{\aleph_0}}$.
\end{enumerate}
The first case occurs if and only if the system $(X,G)$ is tame.
\end{thm}

\br

\begin{defn}
A Banach space $V$ is called {\em Rosenthal} if it does not contain an isomorphic copy of
$\ell_1(\N)$. For a separable Banach space $V$ an equivalent condition is that $\card V = \card V^{**}
= 2^{\aleph_0}$ (see e.g. \cite{OR} and \cite{Dulst}).
Every Asplund space is Rosenthal.
\end{defn}

\br

 Let $V$ be a Banach space. Denote by $\Iso(V)$ the topological
group of all linear isometries of $V$ onto itself equipped with the
pointwise convergence topology.

\begin{defn} \label{d:repr} \cite{Meg} Let $X$ be a $G$-space.
A \emph{representation} of $(X,G)$ on a Banach space
$V$ is a pair
$$
h  \colon G \to {\Iso}(V), \ \ \ \alpha \colon X \to V^*
$$
where
$h \colon  G \to {\Iso}(V)$ is a continuous co-homomorphism and
$\alpha \colon X \to V^*$ is a weak$^*$ continuous bounded
$G$-map with respect to the {\em dual action\/} $G \times V^* \to V^*,
\ (g\varphi)(v):=\varphi(h(g)(v))$.
We say that a representation $(h,\alpha)$ is
\emph{faithful} when $\alpha$ is a topological embedding.
\end{defn}

Every compact $G$-space $X$ admits a canonical faithful representation on the Banach space $V =C(X)$
via the map $x \mapsto \del_x \in C(X)^*$. 
A natural program is then to classify dynamical systems according to their representability 
properties on ``nice" Banach spaces.
In the following table we encapsulate some features of the
trinity:  a dynamical system $(X,G)$, its enveloping semigroup $E(X)$,
and a class of Banach spaces on at least one of its members the dynamical system $(X,G)$
can be faithfully represented.

\br

Let $X$ be a compact metrizable $G$-space and
$E(X)$ denote the corresponding enveloping semigroup. The symbol
$f$ stands for an arbitrary function in $C(X)$ and
$fG = \{f \circ g: g\in G\}$ denotes its orbit.
Finally, $\cls(fG)$ is the pointwise closure of $fG$ in
$\R^X$.
For more details on this classification see e.g. the review \cite{GM-R}.
WAP stands for weakly almost periodic and HNS for hereditarily not sensitive.


{\tiny{
\begin{table}[h]
\begin{center}
\begin{tabular}{ | l | l| l | l | l | l | }
\hline  & Dynamical characterization &  Enveloping semigroup &
Banach representation\\
\hline  WAP & $\cls(fG)$ is a subset of $C(X)$  &
Every element is continuous & Reflexive \\
\hline HNS & $\cls(fG)$ is metrizable & $E(X)$ is metrizable
& Asplund\\
\hline Tame &  $\cls(fG)$ is Fr\'echet &
Every element is Baire 1 & Rosenthal\\
\hline
\end{tabular}

\br
\caption{ \protect  The hierarchy of Banach representations}
\end{center}
\end{table}
}}

We have the following theorem (\cite[Theorems 6.3 and Theorem 6.9]{GM-rose})

\begin{thm} \label{t:general}
Let $X$ be a compact $G$-space, $F \subset C(X)$ a Rosenthal
family for $X$ such that
$F$ is $G$-invariant (that is, $f\circ g  \in F, \ \ \forall f  \in F, \ \forall g \in G$). 
\ben
\item
There exist: a Rosenthal Banach space $V$, an injective mapping
$\nu: F \to B_V$ into the unit ball $B_V$ of $V$ and a continuous representation
$$
h \colon G \to \Iso(V), \ \ \ \al \colon
X \to V^*
$$
of $(X, G)$ on $V$ 
($\al$ is a topological embedding if $F$ separates points of $X$) and
$$
f(x)= \langle \nu(f), \al(x)
\rangle \ \ \ \forall \ f \in F \ \ \forall \ x \in X.
$$
Thus the following diagram commutes

\begin{equation*}
\xymatrix{ F \ar@<-2ex>[d]_{\nu} \times X
\ar@<2ex>[d]^{\al} \ar[r]  & \R \ar[d]^{id_{\R}} \\
V \times V^* \ar[r]  &  \R }
\end{equation*}

\item
%
If X is metrizable then $V$ is separable.
\een
\end{thm}

\br

In order to see that this theorem can produce the required counterexample we will have to look 
into some details of these constructions, as follows.


\br

For brevity of notation let $\Acal := C(X)$ denote the Banach
space $C(X)$ where,
$\|\cdot\|$ will denote the sup-norm on $C(X)$, $B$ will denote its unit ball,
and $B^*$ will denote the weak$^*$ compact unit ball of the dual space $\Acal^*=C(X)^*$.

\br

Let $W$ be the symmetrized convex hull of $F$; that is,
$$W:=\conv (F \cup -F).$$
It is shown in \cite{GM-rose} that $W$ is a Rosenthal family for $X$,
and also a Rosenthal family for the larger space $B^*$ (where $W$ is considered as a set of
functions on $B^*$).

%
%
%
%
%

\br

Consider the sequence of sets 
$$
M_n:=2^n W + 2^{-n} B.
$$
Since $W$ is convex and
symmetric we can apply the construction of
Davis-Figiel-Johnson-Pelczy\'nski
\cite{DFJP}
as follows.
Let $\| \ \|_n$ be the Minkowski
functional of the set $M_n$, that is,
$$
\| v\|_n = \inf\ \{\lambda
> 0  :  \ v\in \lambda M_n\}.
$$
Then $\| \ \|_n$ is a norm on $\Acal$ equivalent to the given norm
of $\Acal$. For $v\in \Acal,$ set
$$
N(v):=\left(\sum^\infty_{n=1} \| v \|^2_n\right)^{1/2},
$$
and let
$$
 V: = \{ v \in \Acal : N(v) < \infty \}.
$$
Also let 
$$
B_V = \{v \in V : N(v) \leq 1\}, \quad {\text{and}} \quad S_V = \{v \in V : N(v) =1\},
$$
and denote by $j: V \hookrightarrow \Acal$ the inclusion map. 

\begin{claim}\label{claim1}
$(V,N)$ is a Banach space and $j: V \to \Acal$ is a continuous linear
injection, with
$$
W \subset j(B_V)=B_V.
$$
\end{claim}

\begin{proof}
This is proved in the original DFJP paper. To see the last assertion note that
if $v \in W$ then $2^nv \in M_n$, hence $\| v\|_n \leq
2^{-n}$ and $N(v)^2 \leq \sum_{n=1}^\infty   2^{-2n} \allowbreak = 1$.
\end{proof}

\br

The given action $G \times X \to X$ induces a natural linear norm preserving
continuous right action $C(X) \times G \to C(X)$ on the Banach space $\Acal=C(X)$.
It follows by the above construction that $W$ and $B$ are $G$-invariant subsets in $\Acal$.
This implies that $V$ is a $G$-invariant subset of $\Acal$ and the restricted natural linear action
$V \times G \to V, \ \ (v,g) \mapsto vg$ is norm preserving, that is, $N(vg)=N(v)$.
Therefore, the co-homomorphism $h: G \to {\Iso}(V), \
h(g)(v):=vg$ is well defined.

Let $j^*: \Acal^* \to V^*$ be the adjoint map of $j: V \to \Acal$.
Define $\al: X \to V^*$ as follows. For every $x \in X \subset
C(X)^*$ set $\al(x)=j^*(\del_x)$. Then $(h,\al)$ is a 
representation of $(X,G)$ on the Banach space $V$.

\br

By the construction $F \subset W \subset B_V$.
Define $\nu: F \hookrightarrow B_V$ as the natural inclusion.
Then
\begin{equation*} \label{F}
f(x)= \langle \nu(f), \al(x) \rangle \ \ \ \forall \ f \in F, \ \ \forall \ x \in X.
\end{equation*}
(We will write this more simply as $f(x) =\al(x)(f)$.)

It follows in particular that if $F$ separates points of $X$ then $\al$ is an embedding.

\br

\begin{claim}\label{claim2}
$B_V \subset \bigcap_{n \in \N} M_n =
\bigcap_{n \in \N} (2^n W + 2^{-n}B)$. 
\end{claim}

\begin{proof}
If $\| v\| <1$ then $\| v\|_n < 1$ for all $n \in \N$ and, 
as the sets $M_n$ are  convex, this implies that $v \in M_n$ for every $n$.
If $\| v\| =1$ we must have $\| v\|_n < 1$ for all $n \in \N$ and again we conclude that
 $v \in M_n$ for every $n$.
\end{proof}

\br

One more important ingredient we will need is the following refinement of the construction of $V$ 
 \cite[Lemma 17.(2) (Lemma 4.4.(2) in the arXiv version)]{GM-OC}, which in turn relies on \cite[Lemma 1.2.2]{Fa}.

\begin{lem}\label{fabian}
Consider the injective map $j \colon V \to C(X)$ and let $\al : =j^* \circ \del \colon X \to V^*$
(where $\del(x) = \del_x$). 
Then 
\begin{enumerate}
\item
The image of $j^*$ is norm dense in $V^*$:
$$
\ol{j^*( C(X)^*)} = V^*.
$$
\item
The image $\al(X)$ is a $w^*$-generating subset of $V^*$;
i.e. 
$$
\spann(\overline{\conv}^{w^*}(\al(X))),
$$ 
is norm dense in $V^*$.
\end{enumerate}
 \end{lem} 

For the reader's convenience we reproduce the proof of Lemma \ref{fabian}
in Appendix \ref{App} below.

\br

With this background at hand we now proceed with our construction as follows:
Given a positive $\del <1$,  let $n_0$ satisfy 
\begin{equation}\label{n0}
2^{-n_0} < \frac{\del}{4}.
\end{equation}
Let $\ep = \frac{1}{10}\del 2^{-n_0}$.

\br

%
%
%
%
%
%


We will next construct a suitable Rosenthal family $F \subset X$.
Choose a function $f \in C(X)$ such that 
$1- \ep \leq f(x) \le 1$ for every $x \in X$ and such that
the values $1$ and $1-\ep$ are attained by $f$.
Let $F_0 = \{f \circ g : g \in G\}$.
We can assume that the set $F_0$ separates points on $X$ (otherwise
we will replace the system $(X,G)$ by the factor which $f_0$ generates).
Let $F : = {\rm norn}{\text{-}}\cls F_0$.
Since our system $(X,T)$ is tame, the family $F_0$ is a $G$-invariant Rosenthal family
and therefore so is $F$.


We now use the set $F$ to create $W: =\conv (F \cup -F)$ and $V$ as above.

%
%
%

\begin{lem}\label{t}
There is a number $t > 0$ such that
$$
\|\al(x)\|_{V^*} =t, \qquad \forall x \in X.
$$
\end{lem}

\begin{proof}
Let $x_0 \in X$ and let $\|\al(x_0)\|_{V^*} = t$. Then, by minimality of $(X,G)$, 
the set $\{gx_0 : g \in G\}$ is dense in $X$ and
it follows  that $\|\al(x)\|_{V^*}  \leq t$ for every $x \in X$. 
Since the same argument applies to any $x \in X$,
we conclude that indeed $\|\al(x)\|_{V^*} = t$ for every $x \in X$.
Since the set $\al(X)$ generates $V^*$ we cannot have $t =0$.
\end{proof}

\begin{lem}
$$
\max \{|w(x) - w(y)| : x, y \in X\} \leq \ep
$$
for every $w \in W$.
\end{lem}

\begin{proof}
It suffices to show that this inequality holds for 
functions of the form $w = \sum_{j=1}^N p_j f_{n_j}$,
where $\sum_{j=1}^N p_j =1,\  0 < p_j$, and $f_{n_j} \in \pm F_0$, for
$j =1,2,\dots,N$.
Now for such $w$ and $x, y \in X$,
\begin{align*}
|w(x) - w(y)| &= \left|\sum_{j=1}^N p_j f_{n_j}(x) - \sum_{j=1}^N p_j f_{n_j}(y)\right|\\
&= \left| \sum_{j =1}^N p_j (f_{n_j}(x) - f_{n_j}(y)) \right| \\
& \leq \sum_{j =1}^N p_j |f_{n_j}(x) - f_{n_j}(y)| \\
& \leq \sum_{j =1}^N p_j\ep = \ep.
\end{align*}
\end{proof}

\begin{lem}\label{alpha}
$$
\|\al(x) - \al(y) \|_{V^*}  \leq \del
$$
for every $x, y \in X$.
\end{lem}

\begin{proof}
By definition
\begin{align*}
\|\al(x) - \al(y) \|_{V^*} & = \sup_{v \in S_V} | (\al(x) - \al(y))(v)| \\
& =\sup_{v \in S_V} |\al(x)(v) - \al(y)(v)|\\
& = \sup_{v \in S_V} |v(x) - v(y)|.
\end{align*}
Now, $S_V \subset \bigcap_{n \in \N} (2^n W + 2^{-n}B)$
hence, a fortiori, $S_V \subset 2^{n_0} W + 2^{-n_0}B$
(see equation (\ref{n0}) above).
Writing $v \in S_V$ as $v =2^{n_0}  w  + 2^{-n_0} b$, with $w \in W$ and $b \in B$, we have
\begin{align*}
|v(x) - v(y)| & = | 2^{n_0}  w(x)  + 2^{-n_0} b(x) - (2^{n_0}  w(y)  + 2^{-n_0} b(y))  | \\
& \leq 2^{n_0} |w(x) - w(y)| + 2^{-n_0} |b(x) - b(y)|\\
& \leq 2^{n_0}\ep + 2^{-n_0}\cdot 2 \\
& \leq 2^{n_0} (\frac{1}{10}\del 2^{-n_0}) + 2 \frac{\del}{4}  \leq \del.
\end{align*}
\end{proof}

\br

We are now ready to present our counterexample to question \ref{conj-K}.

\begin{thm}\label{thm-CE}
For $G = F_2$, the free group on two generators and for every $0 < \del < 1$
there exists a separable Rosenthal Banach space $V$ and a representation
$$
h \colon G \to {\Iso}(V), \ \ \ \alpha \colon X \to V^*
$$
such that 
\begin{enumerate}
\item
The only $G$ fixed point in $V^*$ is $0$.
\item
There exists an element $\xi \in V^*$, $\|\xi\|=1$ 
such that $\|g\xi - \xi\| \leq \del$ for every $g \in G$.
\end{enumerate}
\end{thm}

\begin{proof}
We consider the representation of the minimal strongly proximal and tame system $(X,G)$
$$
h  \colon  G \to \Iso(V), \ \ \ \al  \colon 
X \to V^*
$$
on the Rosenthal Banach space $V$ described above.
(Again we note that, since $(X,G)$ is tame, our family $F \subset C(X)$ is a Rosenthal family,
although this fact plays no part in the proof; see Remark \ref{no-need} below.)

(1) \ 
The continuous map $j^* \colon  C(X)^* \to V^*$
intertwines the $G$-actions on these spaces.
Let $Q: = \overline{\conv}^{w^*}(\al(X))$. As a homomorphic image of 
$(\overline{\conv}^{w^*}(X),G)$ the system $(Q,G)$
is a strongly proximal affine system. Moreover, the restriction of
$j^*$ to $\{\del_x : x \in X\}$, namely the function $\al$, is an
isomorphism  $\al \colon (X,G) \to (\al(X),G)$.

Suppose now that $\xi \in V^*$ is a nonzero $G$ fixed point.
Normalizing we can assume that $\|\xi \|=1 = t$ (see Lemma \ref{t}).

By Lemma \ref{fabian}, 
i.e. by the norm density of $\spann(Q)$, given $\ep >0$ there are
real numbers $a_1, a_2, \dots, a_k$ and elements $\theta_1, \theta_2, \dots,\theta_k$ in $Q$
such that 
$$
\left\|\sum_{i=1}^k a_i \theta_i - \xi \right\| \leq \ep.
$$ 
By strong proximality there is a  sequence $g_n$ in $G$, and $z$ a point in $X$, such that 
$$
\lim g_n \theta_i = \al(z), \qquad \forall \ \ 1 \leq i \leq k,
$$
(to see this use induction on $k$ and the fact that $X$ is the unique minimal subset of $Q$).
As $G$ acts by norm isometries we also have
$$
\left\|\sum_{i=1}^k a_i g_n \theta_i - \xi \right\| \leq \ep, \qquad \forall g_n.
$$ 
As the closed ball of radius $\ep$ is $w^*$-compact, passing to the limit,  we conclude that
$$
\left\|\left(\sum_{i=1}^k a_i \right)\al(z)  - \xi \right\| \leq \ep.
$$
Note that, as both $\|\xi\| =1$ and $\|\al(z)\|=1$, this implies that
$|\sum_{i=1}^k a_i| \leq 1+ \ep$.
Finally, as $\ep$ was arbitrary we conclude, by compactness, that 
for some $z \in Z$ we have $\al(z) = \pm\xi$, contradicting the fact that
the system $(\al(X),G)$ is not trivial.

(2) \ By Lemma \ref{alpha} we have $\|\al(x) - \al(y)\| \leq \del$, for every $x, y \in X$.
Thus for any $x \in X$ we get, with $\xi =\al(x)$, 
$$
\|g\xi - \xi\| = \| g\al(x) - \al(x)\| =
\|\al(gx) - \al(x)\| \leq \del, \quad \forall g \in G.
$$
\end{proof}

\begin{rmk}\label{no-need}
Our proof does not rely on the full force of Theorem \ref{t:general}; namely
the fact that the resulting Banach space $V$ is Rosenthal is not needed.
All we need from the DFJP construction are the Claims \ref{claim1} and \ref{claim2}. 
\end{rmk}

\br

\begin{rmk}
A variant of the question \ref{conj-K}, also suggested by Kazhdan and Yom Din, is as follows:

\begin{question}\label{conj-K2}
For any Banach space $V$, equipped with a linear and isometric action of a discrete group $G$, 
there exists a positive function $\ep (\delta), \ 0 < \del < 1$, such for any 
$\alpha \in V^*, \ \|\al\|=1$, such that $\| g(\alpha)- \alpha \|\leq \ep (\delta ), \
\forall g\in G$, there exists a G-invariant $\beta \in V^*$ such that $\| \beta- \alpha \|\leq \delta $.
\end{question}


Now we can easily tweak our construction to refute this latter question as well.  
Applying Theorem \ref{thm-CE} with $\del = 1/n$, let us denote by $V^*_n$ the resulting Banach
space with its $F_2$ action. 
Let ${\mathbf{W}} = \bigoplus_{n \in \N} V^*_n$ be the $\ell_2$-sum
of these Banach spaces.
This is a Rosenthal space and it is also true that 
$$
\left( \bigoplus_{n \in \N} V_n \right)^* = \bigoplus_{n \in \N} V^*_n 
$$
(see \cite[Lemma 3 (Lemma 1.14 in the arXiv version)]{GM-OC}).
Let $G$ act diagonally on ${\mathbf{W}}$. 
It is then easily checked that for this action of $G$ on ${\mathbf{W}}$,
$0$ is the unique fixed point, and that the assertion of Question \ref{conj-K2} does not hold.
\end{rmk}

%
%
%
%
%
%

\br

\section{Appendix A: A proof of Lemma \ref{fabian}}\label{App}

Consider a Banach space $(W, \| \cdot\|)$ and a sequence of equivalent norms $\{\|\cdot\|_n\}_{n=1}^\infty$ 
on $W$. 
Let $Z$ be the Banach space
$$
Z = (\sum_{n=1}^\infty W, \| \cdot\|_n)_{\ell_2},
$$ 
that is
$$
Z = \{(w_1,w_2,\dots) \in W^\N : \sum_{n=1}^\infty \|w_n\|^2_n < \infty\},
$$
and norm $\vertiii{\cdot}$ on $Z$ defined by
$$
\vertiii{(w_1,w_2,\dots)}= \left( \sum_{n=1}^\infty \|w_n\|^2_n \right)^{1/2}.
$$
Let $Y \subset Z$ be the subspace
$$
\{(w,w,\dots) \in Z : w \in W\}
$$
and define $T : Y \to W$ by $T((w,w,\dots)) = w$.

\begin{lem}\label{dense}\cite[Lemma 1.2.2]{Fa}
The map $T$ is linear, injective and continuous and $T^* W^*$ is dense in $Y^*$.
\end{lem}

\begin{proof}
Clearly $T$ is linear and injective. The continuity of $T$ follows from the fact that $\| \cdot \|_1$ is equivalent
to $\| \cdot \|$.
Let $y^* \in Y^*$ be given, By the Hahn-Banach theorem there is $z^* \in Z^*$ such that
$z^* \rest Y = y^*$.
Using the definition of $Z$ we can find elements $\xi_n \in W^*$ such that
$\sum_{n=1}^\infty \|\xi_n\|^2_n < \infty$ 
(here we use the same symbol for a norm and its dual norm),
and such that for every
$y = (w,w,\dots) \in B_Y$,
$$
\langle y^*, y \rangle = \langle z^*, y\rangle = \langle z^*, (w,w,\dots)\rangle =
\sum_{n=1}^\infty \langle \xi_n,w \rangle.
$$
It follows that for $m =1,2,\dots$,
\begin{align*}
\langle y^* - T^*(\sum_{n=1}^m \xi_n), y \rangle & =
\langle y^*, y \rangle -
\sum_{n=1}^m \langle \xi_n, w \rangle \\
& \leq \sum_{n = m+1}^\infty \|\xi_n\|_n \|w\|_n \\
& \leq
(\sum_{n=m+1}^\infty \|\xi_n\|^2_n )^{1/2} 
(\sum_{n=m+1}^\infty \|w\|^2_n)^{1/2}.
\end{align*}
Thus, as $m \to \infty$
$$
\vertiii {y^* - T^*(\sum_{n=1}^m \xi_n) } \leq \left(\sum_{n=m+1}^\infty \|\xi_n\|^2_n \right)^{1/2}  \to 0.
$$
But $T^*(\sum_{n=1}^m \xi_n)$ belongs to $T^* W^*$. Therefore $\ol{T^* W^*} = Y^*$.
\end{proof}

\br

Now, applying Lemma \ref{dense} to the situation in Lemma \ref{fabian}, with  $W = C(X)$,
we clearly have $Y =V$ and $T =j$.
Thus we have that $\ol{j^*( C(X)^*)} = V^*$ and the assertion of Lemma \ref{fabian}(2) follows easily.

\br

\section{Apendix B by N. Monod: A negative answer in the case of $\Bcal(H)$}

%
%



%
%
%

In the above Eli Glasner proved (among other things) that 
Question~\ref{conj-K} has a negative answer in general, including for some weak-$*$ closed subspaces of 
$\Bcal(H)$, suggesting that \ref{conj-K} also has a negative answer for $\Bcal(H) = V^*$
\footnote{The predual $V$ for $\Bcal(H)$ is the space of trace-classs operators.
See the remark following Question \ref{conj-K}.}.

\begin{thm}
Question~\ref{conj-K} has a negative answer also in the case where $V^* = \Bcal(H)$.
\end{thm}

\emph{Proof}.
We shall use ideas introduced by Bo{\.z}ejko--Fendler~\cite{Bozejko-Fendler91} in the context of Herz--Schur multipliers and uniformly bounded representations. This leads to a concrete and self-contained construction as follows.

Let $G$ be the free group on an infinite set $S$ of generators. For any non-trivial $x\in G$, denote by $x_-$ the element obtained by erasing the right-most letter in the reduced word for $x$. The Hilbert space $H$ for question~\ref{conj-K} will be $H = \ell^2(G)$ with Hilbert basis $(\delta_x)_{x\in G}$. The $G$-action on 
$\Bcal(H)$ is the adjoint representation associated to the left regular representation $\lambda$ on $\ell^2(G)$. That is, for $T\in \Bcal(H)$ and $g\in G$, we define $g.T = \lambda(g)\circ T \circ \lambda(g)^{-1}$. This is indeed isometric and induced by an isometric representation on the predual $V$. Below, all norms are understood in 
$V^*=\Bcal(H)$ unless specified otherwise with a subscript such at $\|\cdot\|_{\ell^2}$.

Given $n>0$, we choose a subset $S_n \subset S$ of $n$ generators. We define $T_n\in \Bcal(H)$ by requiring $T_n \delta_x = \delta_{x_-}$ if $x_{-}^{-1} x \in S_n^{\pm 1}$ (with $x\neq e$) and $T_n \delta_x = 0$ in all other cases; this defines a bounded operator.

Fix $g\in G$.  Then $T_n$ \emph{almost} commutes with $\lambda(g)$ because $(gx)_- = g (x_-)$ holds unless $x$ is completely cancelled by $g$ (or $x$ is already trivial). Let us compute the commutator when this complete cancellation occurs. Write $g$ in reduced form as $g=s_{1}^{\epsilon_1} \cdots s_{k}^{\epsilon_k}$ with $k\in\N$, $s_i\in S$ and $\epsilon_i=\pm 1$ (we can assume $k>0$). Then $x=s_{k}^{-\epsilon_k} \cdots s_{h+1}^{-\epsilon_{h+1}}$ for some $h<k$. Thus $g x = s_{1}^{\epsilon_1} \cdots s_{h}^{\epsilon_h}$. Now we have
\[
\big( T_n \circ \lambda(g) - \lambda(g) \circ T_n \big) \delta_x = \delta_{s_{1}^{\epsilon_1} \cdots s_{{h-1}}^{\epsilon_{h-1}}} - \delta_{s_{1}^{\epsilon_1} \cdots s_{{h+1}}^{\epsilon_{h+1}}}
\]
provided $s_{h}$ and $s_{h+1}$ are in $S_n$ (otherwise the corresponding term simply does not occur). This implies the estimate
\[
\big\| T_n \circ \lambda(g) - \lambda(g) \circ T_n \big\|  \leq2 \kern5mm\forall\, g\in G \ \forall\,n
\]
because, $g$ being fixed, only one $x\in G$ gives rise to the pair of elements ${s_{1}^{\epsilon_1} \cdots s_{{h-1}}^{\epsilon_{h-1}}}$ and ${s_{1}^{\epsilon_1} \cdots s_{{h+1}}^{\epsilon_{h+1}}}$.
In other words, this commutator is a sum of two partial isometries. 
(A more high-tech argument with Riesz--Thorin interpolation can be found in~\cite[\S2]{Bozejko-Fendler91}.)

\medskip

At this point we have elements $T_n\in V^*= \Bcal(H)$ with $\| g. T_n - T_n\|\leq 2$ for all $g\in G$. We now examine the distance
\[
d(T_n, {V^*}^G) := \inf\big\{ \| T_n - \beta \| :  \text{ $\beta$ is $G$-fixed } \big\}
\]
from $T_n$ to any $G$-fixed point in $\Bcal(H)$.

We have $T_n \ch_{S_n^{\pm 1}} = 2n \delta_e$ and use that $\beta$ commutes with $\lambda$ to compute
\begin{align*}
2n = \| T_n \ch_{S_n^{\pm 1}} \|_{\ell^2} &\leq \| \beta \ch_{S_n^{\pm 1}} \|_{\ell^2} + \| T_n - \beta \| \cdot  \| \ch_{S_n^{\pm 1}} \|_{\ell^2}\\
  &= \| \lambda(\ch_{S_n^{\pm 1}}) \beta \delta_e \|_{\ell^2} + \sqrt{2n} \,\| T_n - \beta \| \\
&\leq \| \lambda(\ch_{S_n^{\pm 1}}) \|\cdot   \| \beta \delta_e \|_{\ell^2} + \sqrt{2n} \,\| T_n - \beta \| \\
&\leq \| \lambda(\ch_{S_n^{\pm 1}}) \|\cdot   \| T_n - \beta \| + \sqrt{2n} \,\| T_n - \beta \|, \\
\end{align*}
where at the end we used $T_n \delta_e = 0$. The norm $\| \lambda(\ch_{S_n^{\pm 1}}) \|$ is well-known to be $2\sqrt{2n-1}$, since it is a spectral radius that can be computed by enumerating paths in a tree, see Thm.~3 in~\cite{Kesten59}. To be very precise: we can consider $H=\ell^2(G)$ as a multiple of the regular representation of the subgroup $F_n$ generated by $S_n$ and Kesten's computation takes place in this $\ell^2(F_n)$; the resulting norm coincides with $\| \lambda(\ch_{S_n^{\pm 1}}) \|$ in $\ell^2(G)$. Now the above inequalities imply
\[
d(T_n, {V^*}^G)  > \frac{\sqrt{2n}}{3}\kern5mm \forall\,n.
\]
We have therefore our answer to the original question by defining $\alpha\in V^* = \Bcal(H)$ to be the unit vector corresponding to $T_n$ and letting $n\to\infty$.\qed

%

\begin{thebibliography}{MM}


%


   





\bibitem{BBHP}
U. Bader, R. Boutonnet, C. Houdayer and J. Peterson,
{\em  Charmenability of arithmetic groups of product type},
arXiv:2009.09952.

\bibitem{BGM}
U. Bader,  T. Gelander and N.  Monod,
{\em A fixed point theorem for $L^1$ spaces}, 
Invent. Math. 189 (2012), no. 1, 143--148.

\bibitem{Bour}
N. Bourbaki, 
{\em Espaces vectoriels topologiques. Chapitres 1 ˆ 5.}
 (French) [Topological vector spaces. Chapters 1Ð5] \'{E}l\'{e}ments de 
 math\'{e}matique. [Elements of mathematics] New edition. Masson, Paris, 1981. vii+368 pp.


\bibitem
{Bozejko-Fendler91}
Marek Bo{\.z}ejko and Gero Fendler, \emph{{{H}erz-{S}chur multipliers and
   uniformly bounded representations of discrete groups.}}, Arch. Math.
   \textbf{57} (1991), no.~3, 290--298.


\bibitem{BFT} J. Bourgain, D.H. Fremlin and M. Talagrand,
{\it Pointwise compact sets in Baire-measurable functions}, Amer.
J. of Math., \textbf{100:4} (1977), 845--886.


\bibitem{Dulst}
D. van Dulst, \emph{Characterization of Banach spaces not
containing $l^1$}, Centrum voor Wiskunde en Informatica,
Amsterdam, 1989.

\bibitem{DFJP}
W. J. Davis, T. Figiel, W. B. Johnson and A. Pelczy\'nski, {\em
Factoring weakly compact operators\/}, {\it J. of Funct. Anal.},
{\bf 17} (1974), 311--327.

\bibitem{Fa}
M. Fabian, {\em Gateaux differentiability of convex functions and
topology. Weak Asplund spaces\/},\ Canadian Math.\ Soc.\ Series of
Monographs and Advanced Texts, New York, 1997.

\bibitem{FHH}
M. Fabian, P. Habala, P. H\'{a}jek, V. Montesinos Santaluc'a, J. Pelant, 
and V. Zizler, 
{\em  Functional analysis and infinite-dimensional geometry},
 CMS Books in Mathematics/Ouvrages de MathŽmatiques de la SMC, 
 {\bf 8}. Springer-Verlag, New York, 2001. x+451 pp.
 
\bibitem{F-G}
H. Furstenberg and E. Glasner, 
{\em Stationary dynamical systems},
Dynamical numbers -- interplay between dynamical systems and number theory, 
1--28, Contemp. Math., {\bf 532}, Amer. Math. Soc., Providence, RI, 2010.

\bibitem{Gl-prox}
S. Glasner,
{\em Proximal flows}, 
 Lecture Notes in Mathematics, Vol. {\bf 517}, Springer-Verlag, Berlin-New York, 1976. 
 
 \bibitem
 {arXiv}
E. Glasner, \emph{On a question of {K}azhdan and {Y}om {D}in},
Preprint, arXiv
   2111.07312v1, 2021.
 
 \bibitem{GM}
E. Glasner and M. Megrelishvili, 
\emph{Hereditarily non-sensitive dynamical systems and linear representations}, 
Colloq. Math., \textbf{104}, (2006), no. 2, 223--283.

\bibitem{GM-rose}
E. Glasner and M. Megrelishvili, {\it Representations of dynamical
systems on Banach spaces not containing $\ell_1$},
Trans. Amer. Math. Soc.,  \textbf{364} (2012),  6395--6424.

\bibitem{GM-12}
E. Glasner and M. Megrelishvili, 
{\em On fixed point theorems and nonsensitivity},
Israel J. Math. {\bf 190}, (2012), 289--305.

\bibitem{GM-OC}
E. Glasner and M. Megrelishvili, 
{\em Banach representations and affine compactifications of dynamical systems},
Asymptotic geometric analysis, 75--144, Fields Inst. Commun., {\bf 68}, Springer, New York, 2013.
ArXiv version: 1204.0432. 

 \bibitem{GM-R}
 E. Glasner and M. Megrelishvili, 
 {\em Representations of dynamical systems on Banach spaces},
 Recent progress in general topology. III, 399--470, Atlantis Press, Paris, 2014. 
 
\bibitem{GM-T}
 E. Glasner and M. Megrelishvili,  
 {\em Todorcevi\'{c}' trichotomy and a hierarchy in the class of tame dynamical systems},
 arXiv:2011.04376.
 
 \bibitem{GTWZ}
E. Glasner, T. Tsankov, B. Weiss  and  A. Zucker, 
{\em  Bernoulli disjointness},
arXiv:1901.03406.


\bibitem{KL} D. Kerr and H. Li,
{\it Independence in topological and $C^*$-dynamics}, Math. Ann.
{\bf 338}  (2007), 869--926.

\bibitem
{Kesten59}
Harry Kesten, \emph{Symmetric random walks on groups}, Trans. Amer.
Math. Soc.
   \textbf{92} (1959), 336--354.

\bibitem{Ko}
A. K\"{o}hler, 
{\em Enveloping semigroups for flows}, 
Proc. Roy. Irish Acad. Sect. A {\bf 95}, (1995), 179--191.

\bibitem{Lo}
V. Losert, 
{\em The derivation problem for group algebras},
Ann. Math. (2),  {\bf 168(1)},  (2008),  221--246.

\bibitem{Meg}
M. Megrelishvili, {\em Fragmentability and representations of flows\/}, 
Topology Proceedings, {\bfseries 27:2} (2003), 497--544.
  
\bibitem{N-P}
I. Namioka, and R. R. Phelps, 
{\em Banach spaces which are Asplund spaces},
Duke Math. J. {\bf 42}, (1975), no. 4, 735--750.

\bibitem{OR}
E. Odell and H. P. Rosenthal, {\em A double-dual characterization
of separable Banach spaces containing $l\sp{1}$}, Israel J. Math.,
{\bf 20} (1975), 375--384.

\bibitem{Sa}
S. Sakai, 
{\em $C^*$-algebras and $W^*$-algebras}, 
Reprint of the 1971 edition. 
Classics in Mathematics. Springer-Verlag, Berlin, 1998.

\bibitem{T}
M. Takesaki, 
{\em  Theory of Operator Algebras. I},
Encyclopaedia of Mathematical Sciences, vol. 124. Springer, Berlin (2002). 
Reprint of the first (1979) edition, Operator Algebras and Non-commutative Geometry 5.

\bibitem{V}
W. A. Veech,
{\em A fixed point theorem-free approach to weak almost periodicity\/},
Trans.\ Amer.\ Math.\ Soc.\ {\bfseries 177}, (1973), 353--362.

\bibitem{YH}
K. Yosida and E. Hewitt, 
{\em  Finitely additive measures},
Trans. Amer. Math. Soc. {\bf 72}, (1952), 46--66. 
 
\end{thebibliography}

\br

\br

\br

Nicolas Monod's address is:
\address{\'Ecole Polytechnique F\'ed\'erale de Lausanne (EPFL)
CH--1015 Lausanne, Switzerland}

\email{nicolas.monod@epfl.ch}


\br

\end{document}